\documentclass{article}

\usepackage{url}
\usepackage{amsmath,amssymb} 
\usepackage{color}
\usepackage[breaklinks=true,colorlinks,bookmarks=false]{hyperref}
\usepackage{verbatim}
\definecolor{Myblue}{rgb}{0,.3,.6}
\newcommand{\emc}[1]{{\textbf{\textit{\color{Myblue}#1}}}}

\DeclareMathOperator{\erf}{erf}

\DeclareMathOperator{\sign}{sign}

\newcommand{\qed}{\hfill \ensuremath{\Box}} 
\newtheorem{theorem}{Theorem}
\newtheorem{lemma}[theorem]{Lemma}
\newtheorem{proposition}[theorem]{Proposition}

\newenvironment{proof}[1][Proof]{\begin{trivlist}
\item[\hskip \labelsep {\bfseries #1}]}{\end{trivlist}}

\newenvironment{example}[1][Example]{\begin{trivlist}
\item[\hskip \labelsep {\bfseries #1}]}{\end{trivlist}}

\begin{document}

\date{}

\author{Hossein Mobahi \\
Computer Science \& Artificial Intelligence Lab.\\
Massachusetts Institute of Technology\\
Cambridge, MA, USA \\
\color{magenta}\texttt{hmobahi@csail.mit.edu} \\
}

\title{Closed Form for Some Gaussian Convolutions}

\maketitle

\begin{abstract}
The convolution of a function with an isotropic Gaussian appears in many contexts such as differential equations, computer vision, signal processing, and numerical optimization. Although this convolution does not always have a closed form expression, there are important family of functions for which closed form exists. This article investigates some of such cases.
\end{abstract}

\section{Introduction}

\paragraph{\bf Motivation.} The convolution of a function with an isotropic Gaussian, aka \emc{Weierstrass transform}, appears in many contexts. In differential equations, it appears as the solution of diffusion equations \cite{Widder75}. In computer vision and image processing, it is used to blur images for generating multi-scale or scale-space representations \cite{scalespace}. In addition, derivatives of the blurred images are used for defining visual operations and features such as in SIFT \cite{SIFT}. In signal processing, it defines filters for noise reduction \cite{filtering} and preventing aliasing in sampled data \cite{aliasing}. Finally, in optimization theory, it is used to simplify a nonconvex objective function by suppressing spurious critical points \cite{GaussConvEnv15}.

\paragraph{\bf Numerical Approximation.} In general, the Weierstrass transform does not have a closed form expression. When that is the case, the transformation needs to be computed or approximated \emc{numerically}. Recent efforts, mainly motivated from optimization applications, have resulted in efficient algorithms for such numerical approximations, e.g. by adaptive importance sampling \cite{approxconv}.

\paragraph{\bf Exact Closed Form.} Putting the general arguments aside, there are important family of functions for which the Weierstrass transform does have a closed form. This article investigates two such families, namely, multivariate polynomials and Gaussian radial basis functions. These two classes are rich, as they are able to approximate any arbitrary function (under some mild conditions), up to a desired accuracy. In addition, we present a result which provides the closed form Weierstrass transform of a multivariate function, given the closed form transform of its univariate cases. 


\paragraph{\bf Notation.} We use $x$ for scalars, $\boldsymbol{x}$ for vectors, $\| \,.\,$ for the $\ell_2$ norm, and $i$ for $\sqrt{-1}$. We also use $\triangleq$ for equality by definition. We define the isotropic Gaussian kernel $k_\sigma(\boldsymbol{x})$ as below.

\begin{equation}
k_\sigma(\boldsymbol{x}) \triangleq (\sqrt{2 \pi} \sigma)^{-\dim(\boldsymbol{x})} e^{-\frac{\| \boldsymbol{x}\|^2}{2 \sigma^2}} \,. \nonumber
\end{equation}

\section{Closed Form Expressions}

\subsection{Polynomials}

Given a multivariate polynomial function $f:\mathbb{R}^n \rightarrow \mathbb{R}$. We show that the convolution of $f$ with $k_\sigma$ has a closed form expression. Suppose $f$ consists of $m$ monomials,

\begin{equation}
f(\boldsymbol{x}) \triangleq \sum_{j=1}^m a_j f_j(\boldsymbol{x})\,, \nonumber
\end{equation}

where $f_j$ is the $j$'th term and $a_j$ is its coefficient. Each monomial has the form  $f_j(\boldsymbol{x}) \triangleq \Pi_{d=1}^n x_d^{p_{d,j}}$ with $x_d$ being the $d$'th component of the vector $\boldsymbol{x}$, and $p_{d,j}$ being a nonnegative integer exponent of $x_d$ in $f_j$. Then we have,

\begin{equation}
[f \star k_\sigma] (\boldsymbol{x}) = \sum_{j=1}^m a_j \big([f_j \star k_\sigma] (\boldsymbol{x}) \big) = \sum_{j=1}^m a_j \Pi_{d=1}^n \Big([\Box_d^{p_{d,j}} \star k_\sigma](x_d) \Big) = \sum_{j=1}^m a_j \Pi_{d=1}^n u(x_d,p_{d,j},\sigma) \,. \nonumber
\end{equation}

The function $u$ has the following form (see the Appendix for the proof),

\begin{equation}
u(x,p,\sigma) \triangleq [\Box^p \star k_\sigma](x) = (i \sigma)^p \,\, h_p(\frac{x}{i \sigma})\,, \nonumber
\end{equation}

where $h_p$ is the $p$'th \emc{Hermite polynomial}, i.e.,

\begin{equation}
h_p(x) \triangleq (-1)^p e^{\frac{x^2}{2}} \frac{d^p}{d x^p} e^{-\frac{x^2}{2}} \,. \nonumber 
\end{equation}

The following table lists $u(x,p,\sigma)$ for $0 \leq p \leq 10$.

\begin{center}
\begin{tabular}{|l|l|}
\hline
$p$ & $u(x,p,\sigma)$ \\
\hline
0 & $1$ \\
1 & $x$ \\
2 & $\sigma^2+x^2$ \\
3 & $3\sigma^2x + x^3$ \\
4 & $3 \sigma^4 + 6 \sigma^2 x^2 + x^4$ \\
5 & $15 \sigma^4 x + 10 \sigma^2 x^3 + x^5$ \\
6 & $15 \sigma^6 + 45 \sigma^4 x^2 + 15 \sigma^2 x^4 + x^6$ \\
7 & $105 \sigma^6 x + 105 \sigma^4 x^3 + 21 \sigma^2 x^5 + x^7$ \\
8 & $105 \sigma^8 + 420 \sigma^6 x^2 + 210 \sigma^4 x^4 + 28 \sigma^2 x^6 + x^8$ \\
9 & $945 \sigma^8 x + 1260 \sigma^6 x^3 + 378 \sigma^4 x^5 + 36 \sigma^2 x^7 + x^9$ \\
10 & $945 \sigma^{10} + 4725 \sigma^8 x^2 + 3150 \sigma^6 x^4 + 630 \sigma^4 x^6 + 
 45 \sigma^2 x^8 + x^{10}$ \\
\hline
\end{tabular}
\end{center}

\paragraph{\bf Example} Let $f(x,y) \triangleq x^2 y^3 + \lambda( x^4+y^4)$. Then, using the table, $[f \star k_\sigma](x,y)$ has the following form,

\begin{equation}
[f \star k_\sigma] = (\sigma^2+x^2)  (3\sigma^2y + y^3) + \lambda( 3 \sigma^4 + 6 \sigma^2 x^2 + x^4 + 3 \sigma^4 + 6 \sigma^2 y^2 + y^4 ) \,. \nonumber
\end{equation}

\subsection{Gaussian Radial Basis Functions}

Isotropic Gaussian Radial-Basis-Functions (RBFs) are general function approximators \cite{RBF90}. That is, any function can be approximated to a desired accuracy by the following form,

\begin{equation}
f(\boldsymbol{x}) \triangleq \sum_{j=1}^m a_j e^{-\frac{\|\boldsymbol{x}-\boldsymbol{x}_j\|^2 }{2 \delta_j^2}}\,, \nonumber
\end{equation}

where $\boldsymbol{x}_j \in \mathbb{R}^n$ is the center and $\delta_j \in \mathbb{R}_{+}$ is the width of the $j$'th basis function, for $j=1,\dots,m$. We have shown in \cite{blur} that the convolution of $f$ with $k_\sigma$ has the following form,

\begin{equation}
[f \star k_\sigma](\boldsymbol{x}) = \sum_{j=1}^m a_j (\frac{ \delta_j }{\sqrt{\delta_j^2+\sigma^2}})^n e^{- \frac{\| \boldsymbol{x} - \boldsymbol{x}_j\|^2 }{2 (\delta_j^2+\sigma^2)}} \,. \nonumber
\end{equation}

In order to be self-contained, we include our proof as presented in \cite{blur} in the appendix of this article.

\subsection{Trigonometric Functions}
\label{sec:trig}
\begin{center}
\begin{tabular}{|l|l|}
\hline
$\sin(x)$ & $e^{-\frac{\sigma^2}{2}} \sin(x)$ \\
$\cos(x)$ & $e^{-\frac{\sigma^2}{2}} \cos(x)$ \\
$\sin^2(x)$ & $\frac{1}{2} \big( 1-e^{-2 \sigma^2} \cos(2 x) \big)$ \\
$\cos^2(x)$ & $\frac{1}{2} \big( 1+e^{-2 \sigma^2} \cos(2 x) \big)$ \\
$\sin(x) \cos(x)$ & $e^{-2 \sigma^2} \sin(x) \cos(x)$ \\
\hline
\end{tabular}
\end{center}

\begin{example}
Let $f(x,y) \triangleq \big( 2 \cos(x) - 3 \sin(y) + 4 \cos(x) \sin(y) - 8 \sin(x) \cos(y) + 4\big)^2$. Compute $[f \star k_\sigma](x,y)$.
\end{example}

We expand the quadratic form and then use the table, as shown below.

\begin{eqnarray}
& & \big( 2 \cos(x) - 3 \sin(y) + 4 \cos(x) \sin(y) - 8 \sin(x) \cos(y) + 4\big)^2 \\
&=& 4 \cos^2(x) + 9 \sin^2(y) + 16 \cos^2(x) \sin^2(y) + 64 \sin^2(x) \cos^2(y) + 16\\
& &  -12 \cos(x) \sin(y) + 16 \cos^2(x) \sin(y) - 32 \cos(x) \sin(x) \cos(y)\\
& & + 16 \cos(x) -24 \sin(y) \cos(x) +48 \sin(y) \sin(x) \cos(y) -24 \sin(y)\\
& & -64 \cos(x) \sin(y) \sin(x) + 32 \cos(x) \sin(y) -64 \sin(x) \cos(y)
\end{eqnarray}

\subsection{Functions with Linear Argument}

This result tells us about the form of $[f \big({\Box}^T \boldsymbol{x} \big) \star k_\sigma(\Box)] (\boldsymbol{w})$ given $[f \star k_\sigma] (y)$. 

{\bf Lemma 1} {\it
Let $f:\mathbb{R} \rightarrow \mathbb{R}$ have well-defined Gaussian convolution $\tilde{f} ( y \,;\, \sigma) \triangleq [f \star k_\sigma] (y)$. Let $\boldsymbol{x}$ and $\boldsymbol{w}$ be any real vectors of the same length. Then $[f \big({\Box}^T \boldsymbol{x} \big) \star k_\sigma(\Box)] (\boldsymbol{w}) = \tilde{f} ( \boldsymbol{w}^T \boldsymbol{x} \,;\, \sigma \| \boldsymbol{x}\|)$.}

We now use this result in the following examples. These examples are related to the earlier Table about the smoothed response functions.
 
\paragraph{\bf Example} Suppose $f(y) \triangleq \sign(y)$. It is easy to check that $\tilde{f} (y;\sigma) = \erf(\frac{y}{\sqrt{2}\sigma})$. Applying the lemma to this implies that $[\sign ({\Box}^T \boldsymbol{x} ) \star k_\sigma(\Box)] (\boldsymbol{w}) = \erf(\frac{\boldsymbol{w}^T \boldsymbol{x}}{\sqrt{2}\sigma \| \boldsymbol{x}\|})$.

\paragraph{\bf Example} Suppose $f(y) \triangleq \max(0,y)$. It is easy to check that $\tilde{f} (y;\sigma) = \frac{\sigma}{\sqrt{2 \pi}} e^{-\frac{y^2}{2 \sigma^2} } + \frac{1}{2} y \big (1+\erf(\frac{y}{\sqrt{2}\sigma}) \big)$. Applying the lemma to this implies that $[\max (0,  {\Box}^T \boldsymbol{x}) \star k_\sigma(\Box)] (\boldsymbol{w}) = \frac{\sigma \| \boldsymbol{x}\|}{\sqrt{2 \pi}} e^{-\frac{(\boldsymbol{w}^T \boldsymbol{x})^2}{2 \sigma^2 \| \boldsymbol{x}\|^2} } + \frac{1}{2} \boldsymbol{w}^T \boldsymbol{x} \big (1+\erf(\frac{\boldsymbol{w}^T \boldsymbol{x}}{\sqrt{2}\sigma \| \boldsymbol{x}\|}) \big)$.

\paragraph{\bf Example} Suppose $f(y) \triangleq \sin(y)$. From trigonometric table in \ref{sec:trig} we know $\tilde{f}(y; \sigma) = e^{-\frac{\sigma^2}{2}} \sin(y)$. Applying the lemma to this implies that $[\sin(\Box^T \boldsymbol{x}) \star k_\sigma(\Box)] (\boldsymbol{w}) = e^{-\frac{\sigma^2 \|\boldsymbol{x}\|^2}{2}} \sin(\boldsymbol{w}^T \boldsymbol{x})$. In particular, when $\boldsymbol{x} \triangleq (x,1)$, then smoothed version of $\sin(w_1 x + w_2)$ w.r.t. $\boldsymbol{w}$ is $e^{-\frac{\sigma^2 (1+x^2)}{2}} \sin(w_1 x + w_2)$.

\bibliography{closed}
\bibliographystyle{apalike}

\newpage~\newpage~

\noindent {\bf \LARGE Appendices}
\bigskip

\appendix

\section{Polynomials}

Here we prove that $[\Box^p \star k_\sigma](x) = (i \sigma)^p \,\, h_p(\frac{x}{i \sigma})$, where $i \triangleq \sqrt{-1}$ and $h_p$ is the $p$'th {\em Hermite} polynomial $h_p(x) \triangleq (-1)^p e^{\frac{x^2}{2}} \frac{d^p}{d x^p} e^{-\frac{x^2}{2}}$. Using Fourier Transform $F$ we proceed as below,

\begin{eqnarray}
& & [\Box^p \star k_\sigma](x) \\
&=& F^{-1} \{ F \{ [\Box^p \star k_\sigma](x) \} \} \\
&=& F^{-1} \{ F \{ x^p \} F\{ k_\sigma(x) \} \} \\
\label{eq:poly_fourier}
&=& F^{-1} \{ \Big( \sqrt{2 \pi} (-i)^p \delta^{(p)}(\omega) \Big) \,\, \Big( \frac{1}{\sqrt{2 \pi}}e^{-\frac{1}{2} \sigma^2 \omega^2} \Big) \} \\
&=& \sqrt{2 \pi} (-i)^p  F^{-1} \{ \Big( \delta^{(p)}(\omega) \Big) \,\, \Big( \frac{1}{\sqrt{2 \pi}}e^{-\frac{1}{2} \sigma^2 \omega^2} \Big) \} \\
&=& \sqrt{2 \pi} (-i)^p  \int_\mathbb{R} \delta^{(p)}(\omega) \,\,  \frac{1}{\sqrt{2 \pi}}e^{-\frac{1}{2} \sigma^2 \omega^2} \,\, e^{i \omega x} \,\, d \omega \\
&=& \sqrt{2 \pi} (-i)^p  \Big(\frac{d^p}{d \omega^p} \big(  \frac{1}{\sqrt{2 \pi}}e^{-\frac{1}{2} \sigma^2 \omega^2 + i \omega x} \big) \Big)_{\omega=0}\\
&=& \sqrt{2 \pi} (-i)^p  \Big(\frac{d^p}{d \omega^p} \big(  \frac{1}{\sqrt{2 \pi}}e^{-\frac{1}{2} \sigma^2 \omega^2 + i \omega x} \big) \Big)_{\omega=0}\\
&=& (-i)^p  \Big(\frac{d^p}{d \omega^p} \big( e^{-\frac{1}{2} \sigma^2 \omega^2 + i \omega x} \big) \Big)_{\omega=0}\\
\label{eq:complete_square}
&=& (-i)^p  \Big(\frac{d^p}{d \omega^p} \big( e^{-\frac{x^2}{2 \sigma^2} - \frac{(\omega - \frac{1}{\sigma^2} i x)^2}{\frac{2}{\sigma^2}}} \big) \Big)_{\omega=0}\\
&=& (-i)^p e^{-\frac{x^2}{2 \sigma^2}} \Big(\frac{d^p}{d \omega^p} \big( e^{ - \frac{(\omega - \frac{1}{\sigma^2} i x)^2}{\frac{2}{\sigma^2}}} \big) \Big)_{\omega=0}\\
\label{eq:hermite}
&=& (-i)^p e^{-\frac{x^2}{2 \sigma^2}} \Big(h_p(\sigma(\omega-\frac{1}{\sigma^2} i x)) (-\sigma)^{p} e^{-\frac{\sigma^2 (\omega-\frac{1}{\sigma^2} i x)^2}{2}} \Big)_{\omega=0}\\
&=& (i \sigma)^p h_p(\frac{x}{i \sigma} ) \,.
\end{eqnarray}

In (\ref{eq:poly_fourier}) we use the fact that $F\{x^p\}=\sqrt{2 \pi} (-i)^p \delta^{(p)}(\omega)$, where $\delta^{(p)}$ denotes the $p$'th derivative of Dirac's delta function. In (\ref{eq:complete_square}) we use completing the square. Finally, note that change of variable $x \rightarrow (x-\mu)/\alpha$ in the Hermite polynomials leads to $h_p(\frac{x-\mu}{\alpha}) \triangleq (-\alpha)^p e^{\frac{(x-\mu)^2}{2 \alpha^2}} \frac{d^p}{d x^p} e^{-\frac{(x-\mu)^2}{2 \alpha^2}}$. In particular, by setting $\alpha = 1/\sigma$ and $\mu=(ix)/\sigma^2$ we obtain $h_p(\sigma(\omega-\frac{1}{\sigma^2} i x)) (-\sigma)^{p} e^{-\frac{\sigma^2 (\omega-\frac{1}{\sigma^2} i x)^2}{2}}  = \frac{d^p}{d \omega^p} e^{-\frac{(\omega-\frac{1}{\sigma^2} i x)^2}{2 \frac{1}{\sigma^2}}}$, which is used in (\ref{eq:hermite}).

\newpage

\section{Gaussian Radial Basis Functions}

In order to prove our main claim, we will need the following proposition.

\paragraph{\bf Proposition 0} {\it The following identity holds for the product of two Gaussians.}

\begin{equation}
k(\boldsymbol{x} - \boldsymbol{\mu}_1 ; \sigma_1^2) \,\, k(\boldsymbol{x} - \boldsymbol{\mu}_2 ; \sigma_2^2) =  \frac{e^{- \frac{\| \boldsymbol{\mu}_1- \boldsymbol{\mu}_2 \|^2 }{2 (\sigma_1^2+\sigma_2^2)}}}{(\sqrt{2\pi(\sigma_1^2+\sigma_2^2)})^m}  k(\boldsymbol{x} - \frac{\sigma_2^2 \boldsymbol{\mu}_1 + \sigma_1^2 \boldsymbol{\mu}_2} {\sigma_1^2+\sigma_2^2}  ; \frac{\sigma_1^2 \sigma_2^2} {\sigma_1^2+\sigma_2^2} ) \nonumber \,.
\end{equation}

\begin{proof}
\begin{eqnarray}
& & k(\boldsymbol{x} - \boldsymbol{\mu}_1 ; \sigma_1^2) \,\, k(\boldsymbol{x} - \boldsymbol{\mu}_2 ; \sigma_2^2) \nonumber\\
&=& \frac{1}{(\sigma_1\sqrt{2\pi})^m} e^{-\frac{\|\boldsymbol{x} - \boldsymbol{\mu}_1\|^2 }{2 \sigma_1^2}}\,\,  \frac{1}{(\sigma_2\sqrt{2\pi})^m} e^{-\frac{\|\boldsymbol{x} - \boldsymbol{\mu}_2\|^2 }{2 \sigma_2^2}} \nonumber \\
&=& \frac{1}{(2\pi \sigma_1 \sigma_2 )^m} e^{-\frac{\|\boldsymbol{x} - \boldsymbol{\mu}_1\|^2 }{2 \sigma_1^2}-\frac{\|\boldsymbol{x} - \boldsymbol{\mu}_2\|^2 }{2 \sigma_2^2}} \nonumber \\
\label{eq:completing_the_square2}
&=& \frac{1}{(2\pi \sigma_1 \sigma_2 )^m} e^{-\frac{\|\boldsymbol{x} - \frac{\sigma_1^2 \sigma_2^2} {\sigma_1^2+\sigma_2^2} (\frac{\boldsymbol{\mu}_1}{\sigma_1^2} + \frac{\boldsymbol{\mu}_2}{\sigma_2^2} )\|^2 }{2 \frac{\sigma_1^2 \sigma_2^2} {\sigma_1^2+\sigma_2^2}} - \frac{\| \boldsymbol{\mu}_1- \boldsymbol{\mu}_2 \|^2 }{2 (\sigma_1^2+\sigma_2^2)}  } \\
&=& \frac{e^{- \frac{\| \boldsymbol{\mu}_1- \boldsymbol{\mu}_2 \|^2 }{2 (\sigma_1^2+\sigma_2^2)}}}{(\sqrt{2\pi(\sigma_1^2+\sigma_2^2)})^m}  k(\boldsymbol{x} - \frac{\sigma_2^2 \boldsymbol{\mu}_1 + \sigma_1^2 \boldsymbol{\mu}_2} {\sigma_1^2+\sigma_2^2}  ; \frac{\sigma_1^2 \sigma_2^2} {\sigma_1^2+\sigma_2^2} ) \nonumber \,.
\end{eqnarray}

Note that (\ref{eq:completing_the_square2}) is derived by completing the square.

\qed

\end{proof}

We are now ready to state the main claim in the following lemma.

\begin{lemma}
Suppose $f=\sum_{j=1}^p a_j e^{-\frac{\|\boldsymbol{x}-\boldsymbol{x}_j\|^2 }{2 \delta_j^2}}$. Then the following identity holds.

\begin{equation}
[f \star k_\sigma] (\boldsymbol{x}) = \sum_{j=1}^p a_j (\frac{ \delta_j }{\sqrt{\delta_j^2+\sigma^2}})^n e^{- \frac{\| \boldsymbol{x} - \boldsymbol{x}_j\|^2 }{2 (\delta_j^2+\sigma^2)}} \,. \nonumber
\end{equation}
\end{lemma}

\begin{proof}
\begin{eqnarray}
[f \star k_\sigma] (\boldsymbol{x}) &\triangleq& \int_{\mathbb{R}^n} f(\boldsymbol{y})  k_\sigma(\boldsymbol{x}-\boldsymbol{y}) \,d \boldsymbol{y} \nonumber \\
&=& \int_{\mathbb{R}^n} \Big( \sum_{j=1}^p a_j e^{- \frac{\| \boldsymbol{y} - \boldsymbol{x}_j \|^2 }{2 \delta_j^2}} \Big)  k_\sigma(\boldsymbol{x}-\boldsymbol{y}) \,d \boldsymbol{y} \nonumber \\
&=& \sum_{j=1}^p a_j  \Big( \int_{\mathbb{R}^n}  e^{- \frac{\| \boldsymbol{y} - \boldsymbol{x}_j \|^2 }{2 \delta_j^2}} k_\sigma(\boldsymbol{x}-\boldsymbol{y}) \,d \boldsymbol{y} \Big) \nonumber \\
&=& \sum_{j=1}^p a_j ( \delta_j \sqrt{2 \pi} )^n \Big( \int_{\mathbb{R}^n} k_{\delta_j}( \boldsymbol{y} - \boldsymbol{x}_j) k_\sigma(\boldsymbol{x}-\boldsymbol{y}) \,d \boldsymbol{y} \Big) \nonumber \\
\label{eq:use_gaussian_product1}
&=& \sum_{j=1}^p a_j ( \delta_j \sqrt{2 \pi} )^n \Big( \int_{\mathbb{R}^n}  \frac{e^{- \frac{\| \boldsymbol{x}_j- \boldsymbol{x} \|^2 }{2 (\delta_j^2+\sigma^2)}}}{(\sqrt{2\pi(\delta_j^2+\sigma^2)})^n}  k_{\sqrt{\frac{\delta_j^2 \sigma^2} {\delta_j^2+\sigma^2}}}(\boldsymbol{y} - \frac{\sigma^2 \boldsymbol{x}_j + \delta_j^2 \boldsymbol{x}} {\delta_j^2+\sigma^2} ) \,d \boldsymbol{y} \Big)    \\
&=&  \sum_{j=1}^p a_j (\frac{ \delta_j }{\sqrt{\delta_j^2+\sigma^2}})^n e^{- \frac{\| \boldsymbol{x}_j- \boldsymbol{x} \|^2 }{2 (\delta_j^2+\sigma^2)}}  \Big( \int_{\mathbb{R}^n} k_{\sqrt{\frac{\delta_j^2 \sigma^2} {\delta_j^2+\sigma^2}}} (\boldsymbol{y} - \frac{\sigma^2 \boldsymbol{x}_j + \delta_j^2 \boldsymbol{x}} {\delta_j^2+\sigma^2} ) \,d \boldsymbol{y} \Big) \nonumber \\
&=& \sum_{j=1}^p a_j (\frac{ \delta_j }{\sqrt{\delta_j^2+\sigma^2}})^n e^{- \frac{\| \boldsymbol{x}_j- \boldsymbol{x} \|^2 }{2 (\delta_j^2+\sigma^2)}} \nonumber \,,
\end{eqnarray}

where in (\ref{eq:use_gaussian_product1}) we use the Gaussian product result from proposition 0.
\qed
\end{proof}

\newpage

\section{Functions with Linear Argument}

In order to prove the lemma, we first need the following proposition.

\begin{proposition}
The following identity holds,

\begin{equation}
[k_\delta(a \,\Box\,+b) \,\star\, k_\sigma (\,\Box\,) ] (x) = k_{\sqrt{\delta^2+a^2 \sigma^2}}(a x+b)\,. \nonumber
\end{equation}

\end{proposition}

\begin{proof}

We first claim that the following identity is true,

\begin{equation}
\label{eq:claim}
\int k_\delta(a t+b) \, k_\sigma (x-t) \,dt = \frac{1}{2} \, k_{\sqrt{\delta^2+a^2 \sigma^2}}(a x+b) \, \erf( \frac{a \sigma^2(at+b) + \delta^2(t-x) }{\delta \sigma \sqrt{2(\delta^2 + a^2 \sigma^2)}} ) + f(a,b,x)\,.
\end{equation}

If the claim holds, the proposition can be easily proved as below,

\begin{eqnarray}
& & [k_\delta(a \,\Box\,+b) \,\star\, k_\sigma (\,\Box\,) ] (x) \\
&=& \int_{-\infty}^{\infty} k_\delta(a t+b) \, k_\sigma (x-t) \,dt \\
&=& \lim_{t \rightarrow \infty}\int k_\delta(a t+b) \, k_\sigma (x-t) \,dt - \lim_{t \rightarrow -\infty}\int k_\delta(a t+b) \, k_\sigma (x-t) \,dt\\
&=& \lim_{t\rightarrow \infty} \frac{1}{2} \, k_{\sqrt{\delta^2+a^2 \sigma^2}}(a x+b) \, \erf( \frac{a \sigma^2(at+b) + \delta^2(t-x) }{\delta \sigma \sqrt{2(\delta^2 + a^2 \sigma^2)}} ) \\
& & - \lim_{t\rightarrow -\infty}\frac{1}{2} \, k_{\sqrt{\delta^2+a^2 \sigma^2}}(a x+b) \, \erf( \frac{a \sigma^2(at+b) + \delta^2(t-x) }{\delta \sigma \sqrt{2(\delta^2 + a^2 \sigma^2)}} )\\
\label{eq:erf_inf}
&=& \frac{1}{2} \, k_{\sqrt{\delta^2+a^2 \sigma^2}}(a x+b) \, (1) - \frac{1}{2} \, k_{\sqrt{\delta^2+a^2 \sigma^2}}(a x+b) \, (-1)\\
&=& k_{\sqrt{\delta^2+a^2 \sigma^2}}(a x+b)\,,
\end{eqnarray}

where (\ref{eq:erf_inf}) uses the fact that $\lim_{x \rightarrow \pm \infty}\erf(x)=\pm 1$. Thus, the following focuses on proving (\ref{eq:claim}), i.e. showing that $k_\delta(a t+b) \, k_\sigma (x-t)$ is anti-derivative of $\frac{1}{2} \, k_{\sqrt{\delta^2+a^2 \sigma^2}}(a x+b) \, \erf( \frac{a \sigma^2(at+b) + \delta^2(t-x) }{\delta \sigma \sqrt{2(\delta^2 + a^2 \sigma^2)}} )$. We do this by differentiating the latter and showing that it equals to the former.

\begin{eqnarray}
& & \frac{d}{dt} \frac{1}{2} \, k_{\sqrt{\delta^2+a^2 \sigma^2}}(a x+b) \, \erf( \frac{a \sigma^2(at+b) + \delta^2(t-x) }{\delta \sigma \sqrt{2(\delta^2 + a^2 \sigma^2)}} ) \\
&=& \frac{1}{2} \, k_{\sqrt{\delta^2+a^2 \sigma^2}}(a x+b) \, \frac{d}{dt} \, \erf( \frac{a \sigma^2(at+b) + \delta^2(t-x) }{\delta \sigma \sqrt{2(\delta^2 + a^2 \sigma^2)}} ) \\
\label{eq:erfp}
&=& \frac{1}{2} \, k_{\sqrt{\delta^2+a^2 \sigma^2}}(a x+b) \, \frac{\sqrt{2(\delta^2+a^2\sigma^2)}}{\sqrt{\pi}\delta \sigma} \, e^{-(\frac{a \sigma^2(at+b) + \delta^2(t-x) }{\delta \sigma \sqrt{2(\delta^2 + a^2 \sigma^2)}})^2} \\
&=& \frac{1}{2} \, \frac{1}{\sqrt{2\pi (\delta^2+a^2 \sigma^2)} } e^{-\frac{(ax+b)^2}{2(\delta^2+a^2 \sigma^2)} } \, \frac{\sqrt{2(\delta^2+a^2\sigma^2)}}{\sqrt{\pi}\delta \sigma} \, e^{-(\frac{a \sigma^2(at+b) + \delta^2(t-x) }{\delta \sigma \sqrt{2(\delta^2 + a^2 \sigma^2)}})^2} \\
&=& \frac{1}{\sqrt{2\pi}\delta} e^{-\frac{(at+b)^2}{2 \delta^2}} \, \frac{1}{\sqrt{2\pi}\sigma} e^{-\frac{(x-t)^2}{2 \sigma^2}}\\
&=& k_\delta(a t+b) \, k_\sigma (x-t) \,,
\end{eqnarray}

where in (\ref{eq:erfp}) we used the identity $\frac{d}{dx} \erf(x)=\frac{2}{\sqrt{\pi}} e^{-x^2}$.

\qed

\end{proof}

Now we are ready to prove the lemma.

\begin{lemma}
Let $f:\mathbb{R} \rightarrow \mathbb{R}$ have well-defined Gaussian convolution $\tilde{f} ( y \,;\, \sigma) \triangleq [f \star k_\sigma] (y)$. Let $\boldsymbol{x}$ and $\boldsymbol{w}$ be any real vectors of the same length. Then $[f \big({\Box}^T \boldsymbol{x} \big) \star k_\sigma(\Box)] (\boldsymbol{w}) = \tilde{f} ( \boldsymbol{w}^T \boldsymbol{x} \,;\, \sigma \| \boldsymbol{x}\|)$.\end{lemma}

\begin{proof}
Using the trivial identity $f\big(h(\boldsymbol{w})\big) = \int_\mathbb{R} f(y) \delta\big(y-h(\boldsymbol{w})\big) \,dy$ for any $h:\mathbb{R}^n \rightarrow \mathbb{R}$, we can proceed as below,

\begin{eqnarray}
& & [f\big(h(\, .\,)\big) \star k_\sigma] (\boldsymbol{w}) \\
&=& \int_{\mathcal{W}} f\big(h(\boldsymbol{t})\big) k_\sigma (\boldsymbol{w} - \boldsymbol{t}) \, d \boldsymbol{t} \\
&=& \int_{\mathcal{W}} \,\,\, \int_\mathbb{R} f(y) \delta\big(y-h(\boldsymbol{t})\big) \,dy \,\,\, k_\sigma (\boldsymbol{w} - \boldsymbol{t}) \, d \boldsymbol{t} \\
\label{eq:conv}
&=& \int_\mathbb{R} f(y) \,\,\, \int_{\mathcal{W}}   \delta\big(y-h(\boldsymbol{t})\big) k_\sigma (\boldsymbol{w} - \boldsymbol{t}) \, d \boldsymbol{t} \,\,\, dy \,.
\end{eqnarray}

We now focus on computing the inner integral $\int_{\mathcal{W}}   \delta\big(y-h(\boldsymbol{t})\big) k_\sigma (\boldsymbol{w} - \boldsymbol{t}) \, d \boldsymbol{t}$. Applying the definition of $h(\boldsymbol{w}) \triangleq \sum_{k=1}^n x_k w_k$ we proceed as the following,

\begin{eqnarray}
& &\int_{\mathcal{W}} \delta\big(y-\sum_{k=1}^n x_k t_k\big) k_\sigma (\boldsymbol{w} - \boldsymbol{t}) \, d \boldsymbol{t} \\
&=&\lim_{\epsilon \rightarrow 0} \int_{\mathcal{W}} k_\epsilon\big(y-\sum_{k=1}^n x_k t_k\big) k_\sigma (\boldsymbol{w} - \boldsymbol{t}) \, d \boldsymbol{t} \\
&=&\lim_{\epsilon \rightarrow 0} \int_{{\mathcal{W}}_n} \,\dots\, \Big( \int_{{\mathcal{W}}_1} k_\epsilon\big(y-\sum_{k=1}^n x_k t_k\big) k_\sigma (w_1 - t_1) \, d t_1 \Big) \,\dots\,  k_\sigma (w_n - t_n)\, d t_n \\
&=&\lim_{\epsilon \rightarrow 0} \int_{{\mathcal{W}}_n} \,\dots\, \Big( k_{\sqrt{\epsilon^2+\sigma^2 x_1^2}} \big(y-x_1 w_1 - \sum_{k=2}^n x_k t_k\big) \Big) \,\dots\,  k_\sigma (w_n - t_n)\, d t_n \\
&=& \lim_{\epsilon \rightarrow 0}  k_{\sqrt{\epsilon^2+\sigma^2 \sum_{k=1}^n x_k^2}} \big( y-\sum_{k=1}^n x_k w_k \big) \\
&=& k_{\sigma \| \boldsymbol{x}\|} \big( y- \boldsymbol{w}^T \boldsymbol{x} \big) \,.
\end{eqnarray}

Plugging this into (\ref{eq:conv}) leads to the identity $ [f\big(h(\, .\,)\big) \star k_\sigma] (\boldsymbol{w}) \\
= \int_\mathbb{R} f(y) k_{\sigma \| \boldsymbol{x}\|} \big( y- \boldsymbol{w}^T \boldsymbol{x} \big) dy$, which can be equivalently expressed as below,

\begin{eqnarray}
& & [f\big(h(\, .\,)\big) \star k_\sigma] (\boldsymbol{w}) \\
&=& \int_\mathbb{R} f(y) k_{\sigma \| \boldsymbol{x}\|} \big( y- \boldsymbol{w}^T \boldsymbol{x} \big) dy \\
&=& [ f \star k_{\sigma \| \boldsymbol{x}\|}] ( \boldsymbol{w}^T \boldsymbol{x} )\\
&=& \tilde{f} ( \boldsymbol{w}^T \boldsymbol{x} \,;\, \sigma \| \boldsymbol{x}\|) \,.
\end{eqnarray}

\qed

\end{proof}

\end{document}